\documentclass[reqno]{amsart}
\usepackage{CJK}
\usepackage{hyperref}
\usepackage{cite}
\usepackage{amsmath}
\usepackage{mathrsfs}
\def\dive{\operatorname{div}}

\numberwithin{equation}{section}

\newtheorem{theorem}{Theorem}[section]
\newtheorem{lemma}[theorem]{Lemma}
\newtheorem{definition}[theorem]{Definition}

\newtheorem{proposition}[theorem]{Proposition}
\newtheorem{remark}[theorem]{Remark}
\newtheorem{corollary}[theorem]{Corollary}

\allowdisplaybreaks

\begin{document}
	
\title[\hfil Equivalence between distributional and viscosity solutions\dots] {Equivalence between distributional and viscosity solutions for the double-phase equation}

\author[Y. Fang and C. Zhang  \hfil \hfilneg]{Yuzhou Fang and  Chao Zhang$^*$}

\thanks{$^*$ Corresponding author.}

\address{Yuzhou Fang \hfill\break School of Mathematics, Harbin Institute of Technology, Harbin 150001, China}
\email{18b912036@hit.edu.cn}

\address{Chao Zhang  \hfill\break School of Mathematics and Institute for Advanced Study in Mathematics, Harbin Institute of Technology, Harbin 150001, China}
\email{czhangmath@hit.edu.cn}


\maketitle

\begin{abstract}
We investigate the different notions of solutions to the double-phase equation
$$
-\dive(|Du|^{p-2}Du+a(x)|Du|^{q-2}Du)=0,
$$
which is characterized by the fact that both ellipticity and growth switch between two different types of polynomial according to the position. 
We introduce the $\mathcal{A}_{H(\cdot)}$-harmonic functions of nonlinear potential theory, and then show that $\mathcal{A}_{H(\cdot)}$-harmonic functions coincide with the distributional and viscosity solutions, respectively. This implies that the distributional and viscosity solutions are exactly the same.
\end{abstract}

\section{Introduction}
\label{sec-1}
Let $\Omega$ be a bounded domain in $\mathbb R^n$, $n\ge 2$. In this work, we are concerned with the relationship of distributional and viscosity solutions to the following double-phase problem
\begin{equation}
\label{main}
-\dive (|Du|^{p-2}Du+a(x)|Du|^{q-2}Du)=0 \quad \text{in } \Omega
\end{equation}
with some appropriate hypotheses, where  $1<p\leq q<\infty$. It is well-known that Eq. \eqref{main} emerges naturally as the Euler-Lagrange equation of the functional
$$
\mathcal{P}(u,\Omega):=\int_\Omega \Big(\frac 1p |Du|^p+\frac{a(x)}{q}|Du|^q\Big)\,dx,
$$
which was first introduced by Zhikov \cite{Zhi86, ZKO94}. The functional $\mathcal{P}$ can provide a model for characterizing the features of strongly anisotropic materials. To be precise, considering two diverse  materials with power hardening exponents $p$ and $q$, the coefficient $a(\cdot)$ determines the geometry of the composite composed of the two, relying on the fact that $x$ belongs to the set $\{a(x)=0\}$ or not. Also the functional $\mathcal{P}$ gives new examples concerning the Lavrentiev phenomenon in \cite{Zhi95,Zhi97}.

Over the last years, these functionals with non-standard growth conditions
$$
W^{1,1}(\Omega)\ni u\mapsto\int_\Omega F(x,u,Du)\,dx,\quad \nu|z|^p\leq F(x,u,z)\leq L(|z|^q+1)
$$
have been a surge of interest. For the case of an autonomous energy density of the type $F(x,u,Du)\equiv  F(Du)$, the regularity theory of minima of such functionals is by now well-understood from the seminal papers of Marcellini \cite{Mar89,Mar91,Mar96}. The study of nonautonomous functionals,
especially for the double-phase problem \eqref{main}, has been continued in a series of nice papers by Mingione et al. For example, the $C^{1,\alpha}$-regularity for local minimizers of $\mathcal{P}$ was derived by Colombo-Mingione in \cite{CM15}. It is shown that under the assumptions that
\begin{equation}
\label{1-2}
0\leq a(\cdot)\in C^{0,\alpha}(\Omega), \quad \alpha\in (0,1] \quad\text{and}\quad \frac{q}{p}<1+\frac{\alpha}{n},
\end{equation}
the minimizers of $\mathcal{P}$ belong to the class $C^{1,\beta}_{\rm loc}(\Omega)$ for some $\beta\in (0,1)$. Baroni-Colombo-Mingione completed the regularity results in \cite{BCM18} for the borderline case $\frac{q}{p}=1+\frac{\alpha}{n}$. Furthermore, a Harnack inequality for minimizers of $\mathcal{P}$ was established  by the same authors in \cite{BCM15}. If the minimizers of $\mathcal{P}$ are bounded, the assumption imposed  on $p, q$ in \cite{CM215} was relaxed to $q\leq p+\alpha$ such that the $C^{1,\beta}_{\rm loc}(\Omega)$ regularity holds.
For the inhomogeneous double-phase equation
$$
-\dive (|Du|^{p-2}Du+a(x)|Du|^{q-2}Du)=-\dive (|F|^{p-2}F+a(x)|F|^{q-2}F) \quad \text{in } \Omega,
$$
Colombo-Mingione \cite{CM16} proved the  Calder\'{o}n-Zygmund type estimate
$$
(|F|^p+a(x)|F|^q)\in L^\gamma_{\rm loc}(\Omega)\Rightarrow(|Du|^p+a(x)|Du|^q)\in L^\gamma_{\rm loc}(\Omega), \quad \gamma>1
$$
under the assumption \eqref{1-2}, and the  estimate above was improved by De Fillippis-Mingione \cite{DeFM} for the borderline case $\frac{q}{p}=1+\frac{\alpha}{n}$.  
Byun-Cho-Oh in \cite{BCO18} obtained the Calder\'{o}n-Zygmund estimates regarding a class of irregular obstacle problems with non-uniformly elliptic operator in divergence form of $(p,q)$-growth. 
For more results, see for instance \cite{Bre12,BCM16,ELM04,BO17} and the references therein.

From the mentioned works above, we can see that rather abundant research results have been derived for double-phase problems. However, there are few results concerning the viscosity solutions for such kind of equations. To this end, our interest in this work focuses on the different notions of solutions to Eq. \eqref{main}. We define naturally the distributional solutions to \eqref{main} based on integration by parts, owing to the divergence form of this equation. At the same time, if the coefficient $a(x)$ is of class $C^1(\Omega)$, then the notion of viscosity solutions is also applicable, which is defined according to pointwise touching test functions. Our aim is to show that the distributional solutions coincide with the viscosity solutions of \eqref{main}. In fact, the equivalence of different notions of solutions is an important topic, which was first investigated by Ishii \cite{Ish95} in the linear case. When it comes to the quasilinear case,  employing the full uniqueness machinery of the theory of viscosity solutions, Juutinen-Lindqvist-Manfredi in \cite{JLM01} verified the equivalence between distributional and viscosity solutions for the $p$-Laplace equation
$$
-\dive (|Du|^{p-2}Du)=0 \quad \text{in } \Omega,
$$
which is the special version of double-phase equation \eqref{main} in the case that $a(x)\equiv0$. Moreover, the equivalence of two different solutions of $p(x)$-Laplace equation
$$
-\dive (|Du|^{p(x)-2}Du)=0 \quad \text{in } \Omega
$$
was obtained by Juutinen-Lukkari-Parviainen in \cite{JLP10}. On the other hand, it is worth to mention that a shorter proof for the equivalence of distributional and viscosity solutions for the $p$-Laplace equation was recently given in \cite{JJ12} by virtue of a technical regularization procedure via infimal convolutions. We refer to \cite{MO19,Sil18} and references therein for elaborate details.

In this paper we revisit the methods developed by Juutinen-Lindqvist-Manfredi \cite{JLM01} relying on the uniqueness of solutions to show that the distributional and viscosity solutions of \eqref{main} coincide. Inspired by the known theory for $p$-harmonic functions and $p(x)$-harmonic functions (see \cite{Fuc90,HHKLM07,HK88,Lin86,Lin17}) which plays an important role in nonlinear potential theory, we introduce a crucial intermediate ingredient  so-called $\mathcal{A}_{H(\cdot)}$-harmonic function.  Then we prove the equivalence of two notions of solutions of \eqref{main} through justifying the distributional and viscosity solutions coincide with $\mathcal{A}_{H(\cdot)}$-harmonic functions, respectively. To the best of our knowledge, the result that distributional solutions and $\mathcal{A}_{H(\cdot)}$-harmonic functions of double-phase equation coincide is new. Our proof relies heavily on the well established theory of distributional solutions to Eq. \eqref{main}, such as the existence, uniqueness and regularity properties.

This paper is organized as follows. In Section \ref{sec-2}, we first state some basic properties of function spaces and notions of solutions. Then we give some auxiliary results that will be used later. Section \ref{sec-3} is devoted to showing that distributional supersolutions and $\mathcal{A}_{H(\cdot)}$-superharmonic functions are the same, and the equivalence between $\mathcal{A}_{H(\cdot)}$-superharmonic functions and viscosity supersolutions is proved in Section \ref{sec-4}. In Section \ref{sec-5}, we verify the comparison principle for viscosity solutions, which is the indispensable element of the proof.

\section{Preliminaries}
\label{sec-2}

In this section, we summarize some basic properties of the Musielak-Orlicz-Sobolev space $W^{1,H(\cdot)}(\Omega)$. In addition, we give the definition of $\mathcal{A}_{H(\cdot)}$-harmonic functions, different notions of solutions to Eq. \eqref{main} together with some auxiliary results.

\subsection{Function spaces}

In the rest of this paper, we assume that
\begin{equation}
\label{2-0}
0\leq a(\cdot)\in C^{0,\alpha}(\Omega), \quad \alpha\in (0,1] \quad\text{and}\quad \frac{q}{p}\leq1+\frac{\alpha}{n}.
\end{equation}
For all $x\in\Omega$ and $\xi\in \mathbb{R}^n$, we shall use the notation
$$
H(x,\xi):=|\xi|^p+a(x)|\xi|^q.
$$
With abuse of notation, we shall also denote $H(x,\xi)$ when $\xi\in\mathbb{R}$. Observe that  $H:\Omega\times[0,+\infty)\rightarrow[0,+\infty)$ is a Musielak-Orlicz function such that satisfies $(\Delta_2)$ and $(\nabla_2)$ conditions (see \cite{BS14,Mus83,Sid13}).

The Musielak-Orlicz space $L^{H(\cdot)}(\Omega)$ is defined as
$$
L^{H(\cdot)}(\Omega):=\left\{u:\Omega\rightarrow \mathbb{R }  \text{ measurable}:\varrho_H(u)<\infty\right\},
$$
endowed with the norm
$$
\|u\|_{L^{H(\cdot)}(\Omega)}:=\inf\left\{\lambda>0:\varrho_H(\frac{u}{\lambda})\leq1\right\},
$$
where $$\varrho_H(u):=\int_\Omega H(x,u)\,dx=\int_\Omega|u|^p+a(x)|u|^q\,dx$$ is called $\varrho_H$-modular.

 The space $L^{H(\cdot)}(\Omega)$ is a separable, uniformly convex Banach space.
From the definitions of $\varrho_H$-modular and norm, we can find
\begin{equation}
\label{2-1}
\min\left\{\|u\|^p_{L^{H(\cdot)}(\Omega)}, \|u\|^q_{L^{H(\cdot)}(\Omega)}\right\}\leq \varrho_H(u)\leq\max\left\{\|u\|^p_{L^{H(\cdot)}(\Omega)}, \|u\|^q_{L^{H(\cdot)}(\Omega)}\right\}.
\end{equation}
It follows from \eqref{2-1} that
$$
\|u_n-u\|_{L^{H(\cdot)}(\Omega)}\rightarrow0 \quad \Longleftrightarrow \quad \varrho_H(u_n-u)\rightarrow0,
$$
which indicates the equivalence of convergence in $\varrho_H$-modular and in norm.

The Musielak-Orlicz-Sobolev space $W^{1,H(\cdot)}(\Omega)$ is the set of those functions $u\in L^{H(\cdot)}(\Omega)$ satisfying their distributional gradients $Du\in L^{H(\cdot)}(\Omega)$. We endow the space $W^{1,H(\cdot)}(\Omega)$ with the norm
$$
\|u\|_{W^{1,{H(\cdot)}}(\Omega)}:=\|u\|_{L^{H(\cdot)}(\Omega)}+\|Du\|_{L^{H(\cdot)}(\Omega)}.
$$
The space $W^{1,H(\cdot)}(\Omega)$ is a separable and reflexible Banach space. The local space $W^{1,H(\cdot)}_{\rm loc}(\Omega)$ is composed of those functions belonging to $W^{1,H(\cdot)}(\Omega')$ for any subdomain $\Omega'$ compactly involved in $\Omega$. Finally, we denote by $W^{1,H(\cdot)}_0(\Omega)$ the closure of $C^{\infty}_0(\Omega)$ in $W^{1,H(\cdot)}(\Omega)$. Indeed, the conditions \eqref{2-0} ensure that the set $C^\infty_0(\Omega)$ is dense in $W^{1,H(\cdot)}(\Omega)$ (see \cite{CDeF20}).

\medskip

The following embedding theorem can be found in \cite{CS16}.

\begin{lemma}
Let $p^*=\frac{np}{n-p}$ if $n<p$ and $p^*=\infty$ otherwise.
\begin{enumerate}
  \item $L^{H(\cdot)}(\Omega)\hookrightarrow L^r(\Omega)$ and $W^{1,H(\cdot)}_0(\Omega)\hookrightarrow W^{1,r}_0(\Omega)$ for all $r\in[1,p]$.

\smallskip

  \item if $n\neq p$, then
  $$
  W^{1,H(\cdot)}_0(\Omega)\hookrightarrow L^r(\Omega) \quad\text{for all } r\in[1,p^*];
  $$
  if $n=p$, then
  $$
  W^{1,H(\cdot)}_0(\Omega)\hookrightarrow L^r(\Omega) \quad\text{for all } r\in[1,\infty).
  $$
  
  \smallskip
  
  \item if $p\leq n$, then
  $$
  W^{1,H(\cdot)}_0(\Omega)\hookrightarrow\hookrightarrow L^r(\Omega) \quad\text{for all }  r\in[1,p^*);
  $$
  if $p>n$, then
  $$
  W^{1,H(\cdot)}_0(\Omega)\hookrightarrow\hookrightarrow L^\infty(\Omega).
  $$
\end{enumerate}
\end{lemma}

Now we give a Poincar\'{e} type inequality coming from Proposition 2.18 in \cite{CS16}.

\begin{lemma}
\label{lem2-1}
If $u\in W^{1,H(\cdot)}_0(\Omega)$, then there is a constant $C>0$, independent of $u$, such that
\begin{equation}
\label{2-2}
\|u\|_{L^{H(\cdot)}(\Omega)}\leq C\|Du\|_{L^{H(\cdot)}(\Omega)}.
\end{equation}
\end{lemma}
For more details about the space $W^{1,H(\cdot)}(\Omega)$, we refer the readers to \cite{CDeF20,CS16,LD18}.
\subsection{Notions of solutions}

Define
$$
A(x,\xi):=|\xi|^{p-2}\xi+a(x)|\xi|^{q-2}\xi,
$$
for all $x\in \Omega$ and $\xi\in\mathbb{R}^n$.  Now we state the diverse type of solutions to  \eqref{main}, and the definition of $\mathcal{A}_{H(\cdot)}$-superharmonic ($\mathcal{A}_{H(\cdot)}$-subharmonic) functions.

\begin{definition} [distributional solution]
\label{def2-1}
A function $u\in W^{1,H(\cdot)}_{\rm loc}(\Omega)$ is called a distributional superslotion to \eqref{main}, if
$$
\int_\Omega\langle A(x,Du), D\eta\rangle\,dx\geq0
$$
for every nonnegative function $\eta\in W^{1,H(\cdot)}_0(\Omega)$. The inequality is converse for distributional subsolution. We say that $u\in W^{1,H(\cdot)}_{\rm loc}(\Omega)$ is a distributional solution of \eqref{main}  if and only if $u$ is both super- and subsolution, that is
$$
\int_\Omega\langle A(x,Du), D\eta\rangle\,dx=0
$$
for each $\eta\in W^{1,H(\cdot)}_0(\Omega)$.
\end{definition}

\begin{definition}[$\mathcal{A}_{H(\cdot)}$-harmonic function]
\label{def2-2}
We say that $u:\Omega\rightarrow(-\infty,\infty]$ is a $\mathcal{A}_{H(\cdot)}$-superharmonic function in $\Omega$, if
\begin{enumerate}
  \item $u$ is lower semicontinuous in $\Omega$;
  
    \smallskip
    
  \item $u$ is finite a.e. in $\Omega$;
  
    \smallskip
    
  \item for any subdomain $D\subset\subset\Omega$ the comparison principle holds: when $h\in C(\overline{D})$ is a distributional solution to \eqref{main}, and $u\geq h$ on $\partial D$, then
      $$
      u\geq h \quad \text{in }  D.
      $$
\end{enumerate}
If $-u$ is $\mathcal{A}_{H(\cdot)}$-superharmonic, then $u:\Omega\rightarrow[-\infty,\infty)$ is called $\mathcal{A}_{H(\cdot)}$-subharmonic function. $\mathcal{A}_{H(\cdot)}$-harmonic function means that it is both $\mathcal{A}_{H(\cdot)}$-superharmonic and $\mathcal{A}_{H(\cdot)}$-subharmonic.
\end{definition}

\begin{definition}[viscosity solution]
\label{def2-3}
A lower semicontinuous function $u:\Omega\rightarrow(-\infty,\infty]$ is a viscosity supersolution of \eqref{main} in $\Omega$, if $u$ is finite a.e. in $\Omega$ and for each $\varphi\in C^2(\Omega)$ such that
\begin{equation}
\label{2-3}
\begin{cases}
\varphi(x_0)=u(x_0)  & x_0\in\Omega,\\
\varphi(x)<u(x) & x\neq x_0,\\
D\varphi(x_0)\neq 0,
\end{cases}
\end{equation}
there holds
\begin{equation}
\label{2-4}
-\dive A(x_0,D\varphi(x_0))\geq0.
\end{equation}
A function $u$ is viscosity subsolution, when $-u$ is a viscosity supersolution of \eqref{main}. A function $u$ is called viscosity solution to \eqref{main} if and only if it is viscosity super- and subsolution.
\end{definition}

\begin{remark}
With the condition \eqref{2-3}, we say that the test function $\varphi$ touches $u$ from below at point $x_0\in \Omega$. In the case that $2\leq p\leq q<\infty$, the equation is pointwise well-defined, so the requirement, $D\varphi(x_0)\neq0$ in \eqref{2-3}, can be eliminated. At this time, $-\dive A(x_0,D\varphi(x_0))=0$.
\end{remark}

\subsection{Auxiliary results}

Now we first provide a comparison principle for distributional solutions. Set $f_+:=\max\{f, 0\}$.
\begin{proposition}[comparison principle]
\label{pro2-4}
Let $u, v\in W^{1,H(\cdot)}_{\rm loc}(\Omega)$ be such that $(u-v)_+\in W^{1,H(\cdot)}_0(\Omega)$. If
$$
\int_\Omega\langle A(x,Du), D\varphi\rangle\,dx\leq \int_\Omega\langle A(x,Dv), D\varphi\rangle\,dx
$$
for each nonnegative $\varphi\in W^{1,H(\cdot)}_0(\Omega)$, then
$$
u\leq v \quad \text{a.e.   in } \Omega.
$$
\end{proposition}

\begin{proof}
Taking $\varphi:=(u-v)_+$ as a testing function to obtain
$$
\int_\Omega\langle A(x,Du)-A(x,Dv), D(u-v)_+\rangle\,dx\leq0.
$$
After straightforward computation, it follows that
$$
u\leq v \quad \text{a.e. in } \Omega.
$$
\end{proof}

\begin{proposition}[Caccioppoli type inequality]
\label{pro2-5}
Let $u$ be a distributional solution to Eq. \eqref{main}. Then for any $\zeta\in C^\infty_0(\Omega)$ with $0\leq\zeta\leq1$, there holds
$$
\int_\Omega\zeta^q H(x,Du)\,dx\leq C(p,q)\int_\Omega H(x,u\cdot D\zeta)\,dx.
$$
\end{proposition}
\begin{proof}
Choosing $\eta:=\zeta^q u$ as a test function in the weak form of Eq. \eqref{main} yields
$$
0=\int_\Omega\langle A(x,Du), D\eta\rangle\,dx.
$$
We further get by Young's inequality with $\varepsilon$ that
\begin{equation*}
\begin{split}
\int_\Omega\zeta^q H(x,Du)\,dx&\leq\int_\Omega\langle A(x,Du),q \zeta^{q-1}u\cdot D\zeta\rangle\,dx\\
&\leq\varepsilon\int_\Omega \zeta^\frac{(q-1)p}{p-1}|Du|^p\,dx+C(\varepsilon)\int_\Omega q^p|u|^p|D\zeta|^p\,dx\\
&\quad+\varepsilon\int_\Omega a(x)\zeta^q|Du|^q\,dx+C(\varepsilon)\int_\Omega q^q a(x)|u|^q|D\zeta|^q\,dx\\
&\leq\varepsilon\int_\Omega \zeta^q H(x,Du)\,dx+q^qC(\varepsilon)\int_\Omega H(x,u\cdot D\zeta)\,dx,
\end{split}
\end{equation*}
where in the last inequality we used the fact that $\frac{(q-1)p}{p-1}\geq q$ and $0\leq\zeta\leq1$. Therefore, taking a suitable value of $\varepsilon$ arrives at
$$
\int_\Omega\zeta^q H(x,Du)\,dx\leq C(p,q)\int_\Omega H(x,u\cdot D\zeta)\,dx.
$$
\end{proof}

We conclude this section by characterizing the viscosity properties of distributional solutions to Eq. \eqref{2-5} that approximates Eq. \eqref{main}.

\begin{lemma}
\label{lem2-6}
Let $a(x)\in C^1(\Omega)$ and $u_\varepsilon\in W^{1,H(\cdot)}(\Omega)$ be a distribution solution to
\begin{equation}
\label{2-5}
-\dive A(x,Du)=\varepsilon.
\end{equation}
Assume that $\varphi\in C^2(\Omega)$ satisfies $u_\varepsilon(x_0)=\varphi(x_0)$, $u_\varepsilon(x)>\varphi(x)$ for all $x\neq x_0$. Furthermore, if $D\varphi(x_0)\neq0$ or $x_0$ is an isolated critical point of $\varphi$, then we arrive at
$$
\limsup_{\stackrel{x\rightarrow x_0}{x\neq x_0}}\,(-\dive A(x,D\varphi(x)))\geq \varepsilon.
$$
\end{lemma}
\begin{proof}
We argue by contradiction. If the claim does not hold, then there exists a constant $r>0$, such that
$$
D\varphi(x)\neq0 \quad \text{and} \quad -\dive A(x,D\varphi(x))<\varepsilon
$$
for $0<|x-x_0|<r$.

Set $0<\rho<r$. For all nonnegative $\eta\in C^\infty_0(B(x_0,r))$, we get
\begin{align*}
-\int_{\rho<|x-x_0|<r}\eta\dive A(x,D\varphi)\,dx=&\int_{\rho<|x-x_0|<r}A(x,D\varphi)\cdot D\eta\,dx\\
&+\int_{|x-x_0|=\rho}\eta A(x,D\varphi)\cdot \frac{x-x_0}{\rho}\,dS.
\end{align*}
We now evaluate
\begin{equation*}
\begin{split}
&\quad \left|\int_{|x-x_0|=\rho}\eta A(x,D\varphi)\cdot \frac{x-x_0}{\rho}\,dS\right|\\
&\leq \|\eta\|_{L^\infty}\int_{|x-x_0|=\rho}|D\varphi|^{p-1}+a(x)|D\varphi|^{q-1}\,dS\\
&\leq C\|\eta\|_{L^\infty}(1+\|D\varphi\|^{q-1}_{L^\infty})\rho^{n-1}\\
&\rightarrow0 \quad \text{when } \rho\rightarrow0.
\end{split}
\end{equation*}
By the counter proposition,
$$
-\int_{\rho<|x-x_0|<r}\eta\dive A(x,D\varphi)\,dx<\varepsilon\int_{\rho<|x-x_0|<r}\eta\,dx\leq\int_{B(x_0,r)}\varepsilon\eta\,dx.
$$
Hence by sending $\rho\rightarrow0$ we know that
$$
\int_{B(x_0,r)}A(x,D\varphi)\cdot D\eta\,dx\leq\int_{B(x_0,r)}\varepsilon\eta\,dx,
$$
which means that $\varphi$ is a distributional subsolution.

Let $$m:=\inf_{x\in\partial B(x_0,r)}(u_\varepsilon-\varphi)>0.$$ Then $\widetilde{\varphi}:=\varphi+m$ is a distributional subsolution as well. Via $\widetilde{\varphi}\leq u_\varepsilon$ on $\partial B(x_0,r)$ and the comparison principle for distributional solutions, we have $\widetilde{\varphi}\leq u_\varepsilon$ in $B(x_0,r)$, but $\widetilde{\varphi}(x_0)>u_\varepsilon(x_0)$. That is a contradiction.
\end{proof}

\section{Distributional solutions coincide with $\mathcal{A}_{H(\cdot)}$-harmonic functions}
\label{sec-3}

In this section, we study the equivalence of distributional supersolutions and $\mathcal{A}_{H(\cdot)}$-superharmonic functions, by symmetry deriving the same relation between subsolutions and $\mathcal{A}_{H(\cdot)}$-subharmonic functions. Eventually, the equivalence can be extended to the distributional solutions and $\mathcal{A}_{H(\cdot)}$-harmonic functions.

We begin with stating that a distributional supersolution is $\mathcal{A}_{H(\cdot)}$-superharmonic after a redefinition in a set of measure zero. The proof is similar to that of $p(x)$-superharmonic functions in \cite[Theorem 6.1]{HHKLM07}. To be self-contained, we repeat the proof here.

\begin{theorem}
\label{thm3-1}
A distributional supersolution to \eqref{main}, $u\in W^{1,H(\cdot)}_{\rm loc}(\Omega)$, is lower semicontinuous (after a possible change in a set of measure zero). We could define
\begin{equation}
\label{3-0}
u(x)=ess\liminf_{y\rightarrow x} u(y)
\end{equation}
pointwise in $\Omega$. Thus $u$ is $\mathcal{A}_{H(\cdot)}$-superharmonic.
\end{theorem}

\begin{proof}
Observe that $u$ is lower semicontinuous and finite almost everywhere. In order to show that $u$ obeys the comparison principle in Definition \ref{def2-2}, we suppose that $D\subset\subset\Omega$ is a subdomain, and $h\in C(\overline{D})$ is a distributional solution in $D$ such that $u\geq h$ on $\partial D$. By the lower semicontinuity, we choose $\varepsilon>0$ and a subdomain $D'\subset\subset D$ such that $u+\varepsilon>h$ in $D\setminus D'$. Because the set $\{u-h\leq -\varepsilon\}$ is closed due to semicontinuity, the function $\min\{u+\varepsilon-h,0\}$ is compactly supported in $D'$. Via the Comparison Principle (Proposition \ref{pro2-4}), $u+\varepsilon\geq h$ a.e. in $D'$. Thus $u+\varepsilon\geq h$ everywhere in $D$ by virtue of \eqref{3-0}, and sending $\varepsilon\rightarrow0$ finishes the proof.
\end{proof}

Next, we are ready to show that $\mathcal{A}_{H(\cdot)}$-superharmonic functions are distributional supersolutions. To this end, consider an obstacle problem of the type
\begin{equation}
\label{3-1}
\mathcal{F}_\psi(\Omega):=\left\{u\in W^{1,H(\cdot)}(\Omega):u\geq \psi \quad \text{a.e. in } \Omega \quad \text{and} \quad u-\psi\in W^{1,H(\cdot)}_0(\Omega)\right\}
\end{equation}
with $\psi\in W^{1,H(\cdot)}(\Omega)$ being an obstacle from below. Also $\psi$ is the boundary value. If a function $u\in \mathcal{F}_\psi(\Omega)$ satisfies
$$
\int_\Omega\langle A(x,Du),D(v-u)\rangle\,dx\geq 0 \quad \text{for any }  v\in\mathcal{F}_\psi(\Omega),
$$
then $u$ is called a solution of \eqref{3-1}.
We can easily find that a solution to the obstacle problem \eqref{3-1} is a distributional supersolution to Eq. \eqref{main}.

We present the following results about existence and regularity properties of problem \eqref{3-1}. We refer the readers to \cite{CDeF20} for more details about the obstacle problem.

\begin{lemma} \label{lem3-2}

Under the assumption \eqref{2-0}, the following two conclusions hold:
\begin{enumerate}
  \item For the obstacle problem \eqref{3-1}, there is a unique solution $u_\psi\in\mathcal{F}_\psi(\Omega)$. Moreover, $u_\psi$ is a distributional supersolution of \eqref{main} in $\Omega$.
      \smallskip
  \item If $\psi\in C(\Omega)\cap W^{1,H(\cdot)}(\Omega)$, then $u_\psi$ is continuous in $\Omega$ and, in the distributional sense, solves Eq. \eqref{main} in the open set $\{x\in\Omega: u_\psi(x)>\psi(x)\}$.
\end{enumerate}
\end{lemma}

Now we utilize this lemma to derive the following approximation result on the $\mathcal{A}_{H(\cdot)}$-superharmonic functions.

\begin{lemma}
\label{lem3-3}
If $u$ is a $\mathcal{A}_{H(\cdot)}$-superharmonic function in $\Omega$, there is an increasing sequence of continuous supersolutions $\{u_j\}$ in domain $D$ satisfying
$$
u=\lim_{j\rightarrow\infty} u_j
$$
pointwise in $D$. Here $D\subset\subset\Omega$ is an arbitrary subdomain.
\end{lemma}
\begin{proof}
By virtue of the lower semicontinuity of $u$, we can obtain a sequence of smooth functions $\psi_j\in C^\infty(\Omega)$ such that
$$
\psi_1(x)\leq\psi_2(x)\leq\cdots\leq\psi_n(x)\leq\cdots \quad \text{and} \quad u(x)=\lim_{j\rightarrow\infty}\psi_j(x)
$$
everywhere in $\Omega$. Fix a regular subdomain $D\subset\subset\Omega$, and denote by $u_j:=u_{\psi_j}$ the solution of \eqref{3-1} in $D$ with $\psi_j$ as an obstacle. Hence $u_j\in\mathcal{F}_{\psi_j}(\Omega)$ and $u_j\geq\psi_j$ in $D$. We claim that
$$
u_1(x)\leq u_2(x)\leq\cdots\leq u_n(x)\leq\cdots \quad \text{and} \quad \psi_j(x)\leq u_j(x)\leq u(x)
$$
at every point $x\in D$. In order to show $u_j\leq u$, we first observe this is true but possibly in the open set $A_j:=\{x\in D:u_j(x)>\psi_j(x)\}$. Through Lemma \ref{lem3-2} $u_j$ is a distributional solution in $A_j$. Because $\psi_j$ and $u_j$ are continuous in $\overline{A_j}$ (the closure of $A_j$), we get $u_j=\psi_j$ on $\partial A_j$. By means of $u_j\leq u$ on $\partial A_j$ and the comparison principle obeyed by $u$ in Definition \ref{def2-2}, it follows that
$$
u_j\leq u \quad \text{in }   A_j.
$$
Thus $u_j\leq u$ in $D$. Analogously, we can justify
$$
u_j\leq u_{j+1}, \quad\quad j=1,2,3,\ldots,
$$
since $u_{j+1}$ obeys the comparison principle for distributional solutions by Lemma \ref{lem3-2}. Consequently, we deduce
$$
u=\lim_{j\rightarrow\infty}\psi_j\leq\lim_{j\rightarrow\infty}u_j\leq u
$$
everywhere in $D$.
\end{proof}

\begin{lemma}
\label{lem3-4}
Let $u$ be a $\mathcal{A}_{H(\cdot)}$-superharmonic function. If $u$ is locally bounded from above in $\Omega$, we can infer that $u\in W^{1,H(\cdot)}_{\rm loc}(\Omega)$ and $u_j$ satisfies
$$
\lim_{j\rightarrow\infty}\int_D H(x,Du-Du_j)\,dx=0,
$$
where $D\subset\subset\Omega$ is a subdomain and $u_j$ is defined as in Lemma \ref{lem3-3}.
\end{lemma}
\begin{proof}
Choose a regular domain $D_1$ such that $D\subset\subset D_1\subset\subset\Omega$. If $u$ is locally bounded from above, then it is bounded in $D_1$. Since $u_j$ ($j=1,2,3,\dots$) is distributional supersolutions ($u_j$ is defined as in Lemma \ref{lem3-3}), it follows from Proposition \ref{pro2-5} that
\begin{equation}
\label{3-2}
\int_D H(x,Du_j)\,dx\leq C(p,q)M^q\int_{D_1}H(x,D\zeta)\,dx=:L, \quad j=1,2,3,\dots,
\end{equation}
where $$M:=\sup_{D_1}u-\inf_{D_1}\psi_1+1,$$ which means that $Du_j$ ($j=1,2,3,\dots$) is uniformly bounded in $L^{H(\cdot)}(D)$. Thereby we know that $Du_j\rightharpoonup Du$ weakly in $L^{H(\cdot)}(D)$ up to a subsequence, and $u\in W^{1,H(\cdot)}(D)$ and $\int_D H(x,Du)\,dx\leq L$. We also deduce $u\in W^{1,H(\cdot)}_{\rm loc}(\Omega)$.

We now justify $$\lim_{j\rightarrow\infty}\int_D H(x,Du-Du_j)\,dx=0.$$  It suffices to show
$$
\lim_{j\rightarrow\infty}\int_{B_r} H(x,Du-Du_j)\,dx=0,
$$
whenever $B_r$ is a ball contained in $D$. In addition, we suppose $B_{2r}\subset\subset D$ is a concentric ball. Let $\zeta\in C^\infty_0(B_{2r}), 0\leq\zeta\leq1$ and $\zeta\equiv1$ in $B_r$. We use $\eta_j:=\zeta(u-u_j)$ as a test function to get
$$
\int_{B_{2r}}\langle A(x,Du_j),D\eta_j\rangle\,dx\geq0.
$$
Then we estimate
\begin{equation}
\label{3-3}
\begin{split}
I_j&:=\int_{B_{2r}}\langle A(x,Du)-A(x,Du_j),D(\zeta(u-u_j))\rangle\,dx\\
&\leq\int_{B_{2r}}\langle A(x,Du),D(\zeta(u-u_j))\rangle\,dx\\
&=\int_{B_{2r}}(u-u_j)\langle A(x,Du),D\zeta\rangle\,dx+\int_{B_{2r}}\zeta\langle A(x,Du),D(u-u_j)\rangle\,dx\\
&\rightarrow 0 \quad \text{as }  j\rightarrow\infty.
\end{split}
\end{equation}
In fact, from $u_j$ converging to $u$ monotonely and $Du_j\rightharpoonup Du$ weakly in $L^{H(\cdot)}(D)$, we can obtain the limit.

On the other hand,
\begin{equation}
\label{3-4}
\begin{split}
I_j&=\int_{B_{2r}}\zeta\langle A(x,Du)-A(x,Du_j),D(u-u_j)\rangle\,dx\\
&\quad+\int_{B_{2r}}(u-u_j)\langle A(x,Du)-A(x,Du_j),D\zeta\rangle\,dx\\
&=:I_{j,1}+I_{j,2}.
\end{split}
\end{equation}
First, for $I_{j,2}$ by H\"{o}lder inequality we arrive at
\begin{equation*}
\begin{split}
I_{j,2}&\leq\int_{B_{2r}}|D\zeta||u-u_j|(|Du|^{p-1}+a(x)|Du|^{q-1}+|Du_j|^{p-1}+a(x)|Du_j|^{q-1})\,dx\\
&\leq \|D\zeta\|_{L^\infty(B_{2r})}\Big[
\left(\int_{B_{2r}}|u-u_j|^p\,dx\right)^\frac{1}{p}\left(\int_{B_{2r}}|Du|^p\,dx\right)^\frac{p-1}{p}\\
&\quad+\left(\int_{B_{2r}}a(x)|u-u_j|^q\,dx\right)^\frac{1}{q}\left(\int_{B_{2r}}a(x)|Du|^q\,dx\right)^\frac{q-1}{q}\\
&\quad+\left(\int_{B_{2r}}|u-u_j|^p\,dx\right)^\frac{1}{p}\left(\int_{B_{2r}}|Du_j|^p\,dx\right)^\frac{p-1}{p}\\
&\quad+\left(\int_{B_{2r}}a(x)|u-u_j|^q\,dx\right)^\frac{1}{q}\left(\int_{B_{2r}}a(x)|Du_j|^q\,dx\right)^\frac{q-1}{q}\Big]\\
&\leq \|D\zeta\|_{L^\infty(B_{2r})}\left(1+\int_{B_{2r}}H(x,Du)\,dx+\int_{B_{2r}}H(x,Du_j)\,dx\right)\\
&\quad\cdot\max_{t\in\{p,q\}}\left(\int_{B_{2r}}H(x,u-u_j)\,dx\right)^\frac{1}{t}.
\end{split}
\end{equation*}
Here we note that the exponents $\frac{1}{p}, \frac{1}{q}, \frac{p-1}{p}$ and $\frac{q-1}{q}$ are less than 1. Utilizing Lebesgue dominated convergence theorem and \eqref{3-2} yields that
\begin{equation}
\label{3-5}
I_{j,2}\rightarrow 0 \quad \text{as } j\rightarrow\infty.
\end{equation}
Since $I_{j,1}\geq0$, combing \eqref{3-3},\eqref{3-4} and \eqref{3-5} we derive
$$
\lim_{j\rightarrow\infty}\int_{B_{2r}}\zeta\langle A(x,Du)-A(x,Du_j),D(u-u_j)\rangle\,dx=0.
$$
Furthermore,
$$
\lim_{j\rightarrow\infty}\int_{B_{r}}\langle A(x,Du)-A(x,Du_j),D(u-u_j)\rangle\,dx=0.
$$

In what follows, we divide the proof into three cases.

\medskip

\textbf{Case 1.} $2\leq p\leq q<\infty$. It is easy to arrive at
\begin{align*}
&\quad\int_{B_{r}}\langle A(x,Du)-A(x,Du_j),D(u-u_j)\rangle\,dx\\
&\geq C\int_{B_{r}} |Du-Du_j|^p+a(x)|Du-Du_j|^q\,dx\\
&=C\int_{B_{r}} H(x,Du-Du_j)\,dx\geq0.
\end{align*}
Thus we get by sending $j\rightarrow\infty$
$$
\int_{B_{r}} H(x,Du-Du_j)\,dx\rightarrow0.
$$

\textbf{Case 2.} $1<p\leq q<2$. For each $\varepsilon\in(0,1]$, we have
\begin{align*}
&\quad \int_{B_{r}}b(x)|Du-Du_j|^t\,dx\\
&\leq C(t)\varepsilon^\frac{t-2}{t}\int_{B_{r}} b(x)\langle|Du|^{t-2}Du-|Du_j|^{t-2}Du_j,Du-Du_j\rangle\,dx\\
&\quad+\varepsilon\int_{B_{r}}b(x)|Du|^t\,dx,
\end{align*}
where $t\in\{p,q\}$ and $b(x)\in\{1,a(x)\}$. Therefore,
\begin{align*}
&\quad \int_{B_{r}} H(x,Du-Du_j)\,dx \\
&\leq C(p,q)\varepsilon^\frac{p-2}{p}\int_{B_{r}}\langle A(x,Du)-A(x,Du_j),Du-Du_j\rangle\,dx\\
&\quad+\varepsilon\int_{B_{r}}H(x,Du)\,dx.
\end{align*}
As $j\rightarrow\infty$, the above inequality becomes
$$
\lim_{j\rightarrow\infty}\int_{B_{r}} H(x,Du-Du_j)\,dx\leq \varepsilon\int_{B_{r}}H(x,Du)\,dx.
$$
Finally, since $\varepsilon>0$ is arbitrary, we infer
$$
\lim_{j\rightarrow\infty}\int_{B_{r}} H(x,Du-Du_j)\,dx=0.
$$

\textbf{Case 3.} $1<p<2\leq q<\infty$. Merging Case 1 and Case 2, we can see
\begin{align*}
&\quad\int_{B_{r}} H(x,Du-Du_j)\,dx\\
&\leq C(p)\varepsilon^\frac{p-2}{p}\int_{B_{r}}\langle|Du|^{p-2}Du-|Du_j|^{p-2}Du_j,Du-Du_j\rangle\,dx+\varepsilon\int_{B_{r}}|Du|^p\,dx\\
&\quad+C\int_{B_{r}} a(x)\langle|Du|^{q-2}Du-|Du_j|^{q-2}Du_j,Du-Du_j\rangle\,dx\\
&\leq C\varepsilon^\frac{p-2}{p}\int_{B_{r}}\langle A(x,Du)-A(x,Du_j),Du-Du_j\rangle\,dx+\varepsilon\int_{B_{r}}|Du|^p\,dx.
\end{align*}
Similarly, we deduce
$$
\lim_{j\rightarrow\infty}\int_{B_{r}} H(x,Du-Du_j)\,dx=0.
$$

In summary, we reach the conclusion.
\end{proof}
Through the approximation theorem, we can easily establish the result that bounded $\mathcal{A}_{H(\cdot)}$-superharmonic functions are distributional supersolutions, which is stated as follows.
\begin{theorem}
\label{thm3-5}
Assume that $u$ is $\mathcal{A}_{H(\cdot)}$-superharmonic and locally bounded in $\Omega$. Then $u\in W^{1,H(\cdot)}_{\rm loc}(\Omega)$ and $u$ is a distributional supersolution to Eq. \eqref{main}, that is,
$$
\int_\Omega \langle A(x,Du),D\eta\rangle\,dx\geq0
$$
for any nonnegative $\eta\in C^\infty_0(\Omega)$.
\end{theorem}
\begin{proof}
We need to verify
$$
\int_\Omega \langle A(x,Du),D\eta\rangle\,dx=\lim_{j\rightarrow\infty}\int_\Omega \langle A(x,Du_j),D\eta\rangle\,dx\geq0,
$$
where $u_j$ is as in Lemma \ref{lem3-3}. It is well known that $u_j$ is distributional supersolution. Now we pass to the limit. We shall employ the elementary vector inequalities:
\begin{equation}
\label{3-6}
||\xi_1|^{t-2}\xi_1-|\xi_2|^{t-2}\xi_2|=\begin{cases}(t-1)|\xi_1-\xi_2|(|\xi_1|^{t-2}+|\xi_2|^{t-2})\quad &\textmd{if }{t\geq2}, \\[2mm]
2^{2-t}|\xi_1-\xi_2|^{t-1}\quad &\textmd{if }{1<t<2},\end{cases}
\end{equation}
where $\xi_1,\xi_2\in \mathbb{R}^n$. We split the proof into three cases.

\medskip

\textbf{Case 1.} $2\leq p\leq q<\infty$. Using H\"{o}lder inequality and \eqref{3-6} we have
\begin{equation*}
\begin{split}
&\quad \int_\Omega\langle A(x,Du)-A(x,Du_j),D\eta\rangle\,dx\\
&\leq \int_\Omega|D\eta||(|Du|^{p-2}Du-|Du_j|^{p-2}Du_j)+a(x)(|Du|^{q-2}Du-|Du_j|^{q-2}Du_j)|\,dx\\
&\leq C\int_\Omega|D\eta||Du-Du_j|[(|Du|^{p-2}+|Du_j|^{p-2})+a(x)(|Du|^{q-2}+|Du_j|^{q-2})]\,dx\\
&\leq C\left(\int_\Omega|D\eta|^p\,dx\right)^\frac{1}{p}\left(\int_\Omega|Du-Du_j|^p\,dx\right)^\frac{1}{p}\left(\int_\Omega(|Du|^{p-2}+|Du_j|^{p-2})^\frac{p}{p-2}\,dx\right)^\frac{p-2}{p}\\
&+C\left(\int_\Omega a(x)|D\eta|^q\,dx\right)^\frac{1}{q}\left(\int_\Omega a(x)|Du-Du_j|^q\,dx\right)^\frac{1}{q}\\
&\quad \cdot\left(\int_\Omega a(x)(|Du|^{q-2}+|Du_j|^{q-2})^\frac{q}{q-2}\,dx\right)^\frac{q-2}{q}\\
&\leq C\left(1+\int_\Omega H(x,D\eta)\,dx\right)\left(1+\int_\Omega H(x,Du)\,dx+\int_\Omega H(x,Du_j)\,dx\right)\\
&\quad \cdot \max_{t\in\{p,q\}}\left(\int_\Omega H(x,Du-Du_j)\,dx\right)^\frac{1}{t}\\
&\rightarrow 0  \quad \text{as }   j\rightarrow\infty,
\end{split}
\end{equation*}
where the limit is inferred by Lemma \ref{lem3-4} and \eqref{3-2}.

\medskip

\textbf{Case 2.} $1<p\leq q<2$. By H\"{o}lder inequality and \eqref{3-6}, then
\begin{align*}
&\quad \int_\Omega|D\eta||(|Du|^{p-2}Du-|Du_j|^{p-2}Du_j)+a(x)(|Du|^{q-2}Du-|Du_j|^{q-2}Du_j)|\,dx\\
&\leq C\int_\Omega|D\eta|(|Du-Du_j|^{p-1}+a(x)|Du-Du_j|^{q-1})\,dx\\
&\leq C\left(\int_\Omega|D\eta|^p\,dx\right)^\frac{1}{p}\left(\int_\Omega|Du-Du_j|^p\,dx\right)^\frac{p-1}{p}\\
&\quad +C\left(\int_\Omega a(x)|D\eta|^q\,dx\right)^\frac{1}{q}\left(\int_\Omega a(x)|Du-Du_j|^q\,dx\right)^\frac{q-1}{q}\\
&\leq C\left(1+\int_\Omega H(x,D\eta)\,dx\right)\cdot \max_{t\in\{p,q\}}\left(\int_\Omega H(x,Du-Du_j)\,dx\right)^\frac{t-1}{t}\\
&\rightarrow 0  \quad \text{as }   j\rightarrow\infty \text{ by Lemma \ref{lem3-4}}.
\end{align*}

\medskip

\textbf{Case 3.} $1<p<2\leq q<\infty$. Merging Case 1 and Case 2 leads to
\begin{align*}
&\quad\int_\Omega|D\eta||(|Du|^{p-2}Du-|Du_j|^{p-2}Du_j)+a(x)(|Du|^{q-2}Du-|Du_j|^{q-2}Du_j)|\,dx\\
&\leq C\left(\int_\Omega|D\eta|^p\,dx\right)^\frac{1}{p}\left(\int_\Omega|Du-Du_j|^p\,dx\right)^\frac{p-1}{p}\\
&\quad+C\left(\int_\Omega a(x)|D\eta|^q\,dx\right)^\frac{1}{q}\left(\int_\Omega a(x)|Du-Du_j|^q\,dx\right)^\frac{1}{q}\\
&\qquad\cdot\left(\int_\Omega a(x)(|Du|^{q-2}+|Du_j|^{q-2})^\frac{q}{q-2}\,dx\right)^\frac{q-2}{q}\\
&\leq C\left(1+\int_\Omega H(x,D\eta)\,dx\right)\left(1+\int_\Omega H(x,Du)\,dx+\int_\Omega H(x,Du_j)\,dx\right)\\
&\qquad \cdot \max_{s\in\{\frac{p-1}{p},\frac{1}{q}\}}\left(\int_\Omega H(x,Du-Du_j)\,dx\right)^s\\
&\rightarrow 0 \quad \text{as } j\rightarrow\infty.
\end{align*}

Therefore, we have
$$
\lim_{j\rightarrow\infty}\int_\Omega\langle A(x,Du)-A(x,Du_j),D\eta\rangle\,dx=0.
$$
Now we finish the proof.
\end{proof}

The combination of Theorems \ref{thm3-1} and \ref{thm3-5} leads to the following conclusion serving as a bridge in the proof of equivalence of viscosity and distributional solutions. It is worth mentioning that the result on the equivalence between distributional solutions and $\mathcal{A}_{H(\cdot)}$-harmonic functions is of independent interest.

\begin{corollary}
\label{cor3-7}
With the condition \eqref{2-0}, a (locally) bounded distributional solution is the same as a (locally) bounded $\mathcal{A}_{H(\cdot)}$-harmonic function in $\Omega$.
\end{corollary}

\section{Equivalence between viscosity solutions and $\mathcal{A}_{H(\cdot)}$-harmonic functions}
\label{sec-4}

To begin with, we shall verify the claim that $\mathcal{A}_{H(\cdot)}$-superharmonic functions are viscosity supersolutions, which can be obtained by the comparison principle for distributional subsolutions and supersolutions (see Proposition \ref{pro2-4}).

\begin{theorem}
\label{thm4-1}
Under the assumption  that $a(x)\in C^1(\Omega)$, the $\mathcal{A}_{H(\cdot)}$-superharmonic functions are the viscosity supersolutions  to \eqref{main}.
\end{theorem}
\begin{proof}
Let $u$ be $\mathcal{A}_{H(\cdot)}$-superharmonic in the domain $\Omega$. If the conclusion does not hold, then we suppose that there exists a test function $\phi\in C^2(\Omega)$ such that for $x_0\in\Omega$,
\begin{equation*}
\begin{cases}
u(x_0)=\phi(x_0)\\
D\phi(x_0)\neq0\\
u(x)<\phi(x) \quad\textmd{for }  x\neq x_0
\end{cases}
\end{equation*}
and it satisfies
$$
-\dive A(x_0,D\phi(x_0))<0.
$$
Owing to continuity, for $x\in B(x_0,\delta)$ with some small $\delta>0$,
$$
D\phi(x)\neq0 \quad \text{and} \quad -\dive A(x,D\phi(x))<0.
$$

Set
$$
m:=\frac{1}{2}\min_{x\in \partial B(x_0,\delta)}\{u(x)-\phi(x)\}>0
$$
and
$$
\widetilde{\phi}:=\phi+m.
$$
Then we can see that $\widetilde{\phi}$ is a distributional subsolution and $\widetilde{\phi}\leq u$ on $\partial B(x_0,\delta)$. It follows from the comparison principle that $\phi(x)+m\leq u(x)$ in $B(x_0,\delta)$,  which contradicts $\phi(x_0)=u(x_0)$.
\end{proof}

Next, we present an essential approximation lemma, which states that the distributional solution to \eqref{main} could be approximated by the solution of \eqref{2-5}.

\begin{lemma}
\label{lem4-2}
Let $u\in W^{1,H(\cdot)}(\Omega)$ be a distributional solution of \eqref{main} and $u_\varepsilon$ a distributional solution to problem \eqref{2-5} with the Dirichlet boundary value $u-u_\varepsilon\in W^{1,H(\cdot)}_0(\Omega)$ with $\varepsilon>0$. Then
$$
u_\varepsilon\rightarrow u \quad \text{locally uniformly in } \Omega,
$$
under the assumption that $$\frac{q}{p}\leq \min\left\{p,1+\frac{\alpha}{n}\right\}.$$
\end{lemma}
\begin{proof}
\textbf{Step 1.} We first obtain the boundedness of $|Du-Du_\varepsilon|$ in $L^{H(\cdot)}(\Omega)$. We utilize $u-u_\varepsilon$ to test the weak formulation for $u_\varepsilon$, which yields that
$$
\int_\Omega\langle A(x,Du_\varepsilon),D(u-u_\varepsilon)\rangle\,dx=\varepsilon\int_\Omega (u-u_\varepsilon)\,dx.
$$
Furthermore, making use of H\"{o}lder inequality, we have
\begin{align*}
&\quad\int_\Omega H(x,Du_\varepsilon)\,dx\\
&=\int_\Omega\langle A(x,Du_\varepsilon),Du\rangle\,dx-\varepsilon\int_\Omega (u-u_\varepsilon)\,dx\\
&\leq \epsilon\int_\Omega H(x,Du_\varepsilon)\,dx+C(\epsilon)\int_\Omega H(x,Du)\,dx+|\Omega|^\frac{p-1}{p}\varepsilon\left(\int_\Omega |u-u_\varepsilon|^p\,dx\right)^\frac{1}{p}.
\end{align*}
Choosing $\epsilon=\frac{1}{2}$ and employing \eqref{2-1} and \eqref{2-2}, we find
\begin{align*}
&\quad \int_\Omega H(x,Du_\varepsilon)\,dx\\
&\leq C\int_\Omega H(x,Du)\,dx+|\Omega|^\frac{p-1}{p}\varepsilon\left(\int_\Omega H(x,u-u_\varepsilon)\,dx\right)^\frac{1}{p}\\
&\leq C\int_\Omega H(x,Du)\,dx+|\Omega|^\frac{p-1}{p}\varepsilon\left(1+\|u-u_\varepsilon\|^q_{L^{H(\cdot)}(\Omega)}\right)^\frac{1}{p}\\
&\leq  C\int_\Omega H(x,Du)\,dx+|\Omega|^\frac{p-1}{p}\varepsilon\left(1+C\|Du-Du_\varepsilon\|^q_{L^{H(\cdot)}(\Omega)}\right)^\frac{1}{p}\\
&\leq  C\int_\Omega H(x,Du)\,dx+C|\Omega|^\frac{p-1}{p}\varepsilon\left(1+\|Du\|^\frac{q}{p}_{L^{H(\cdot)}(\Omega)}+\|Du_\varepsilon\|^\frac{q}{p}_{L^{H(\cdot)}(\Omega)}\right)\\
&\leq  C\left(1+\|Du\|^q_{L^{H(\cdot)}(\Omega)}\right)+C|\Omega|^\frac{p-1}{p}\varepsilon\left(1+\|Du\|^q_{L^{H(\cdot)}(\Omega)}+\|Du_\varepsilon\|^p_{L^{H(\cdot)}(\Omega)}\right),
\end{align*}
where in the last inequality we have used the assumption $\frac{q}{p}\leq p$.

On the other hand,
$$
\|Du_\varepsilon\|^p_{L^{H(\cdot)}(\Omega)}\leq 1+\int_\Omega H(x,Du_\varepsilon)\,dx.
$$
Thus when $\varepsilon$ is small enough, we derive
$$
\|Du_\varepsilon\|^p_{L^{H(\cdot)}(\Omega)}\leq C\left(1+\|Du\|^q_{L^{H(\cdot)}(\Omega)}\right).
$$
Then
\begin{equation}
\label{4-1}
\|Du_\varepsilon\|_{L^{H(\cdot)}(\Omega)}\leq C\left(1+\|Du\|^q_{L^{H(\cdot)}(\Omega)}\right)^\frac{1}{p}\leq C\left(1+\|Du\|^q_{L^{H(\cdot)}(\Omega)}\right).
\end{equation}
Consequently,
$$
\|Du-Du_\varepsilon\|_{L^{H(\cdot)}(\Omega)}\leq C\left(1+\|Du\|^q_{L^{H(\cdot)}(\Omega)}\right).
$$

\textbf{Step 2.} We show that $\|Du-Du_\varepsilon\|_{L^{H(\cdot)}(\Omega)}\rightarrow0$ as $\varepsilon$ tends to 0. Then we can further conclude $u_\varepsilon\rightarrow u$ in $W^{1,H(\cdot)}(\Omega)$.

We take $u-u_\varepsilon$ to test the weak forms of \eqref{main} and \eqref{2-5}, and then subtract these two equations, which leads to
\begin{equation}
\label{4-2}
\int_\Omega\langle A(x,Du)-A(x,Du_\varepsilon),D(u-u_\varepsilon)\rangle\,dx=-\varepsilon\int_\Omega (u-u_\varepsilon)\,dx.
\end{equation}
As in Step 1, the right-hand side in the last display can be evaluated:
$$
\varepsilon\left|\int_\Omega (u-u_\varepsilon)\,dx\right|\leq \varepsilon C\left(1+\|Du\|^{q^2}_{L^{H(\cdot)}(\Omega)}\right).
$$

Next we focus on the term of the left-hand side in \eqref{4-2}.

\medskip

\textbf{Case 1.} If $2\leq p\leq q<\infty$, by the elementary vector inequality we derive
\begin{align*}
& \quad C\int_\Omega H(x,Du-Du_\varepsilon)\,dx\\
&\leq \int_\Omega\langle A(x,Du)-A(x,Du_\varepsilon),D(u-u_\varepsilon)\rangle\,dx\\
&\leq \varepsilon C\left(1+\|Du\|^{q^2}_{L^{H(\cdot)}(\Omega)}\right)\\
& \rightarrow 0 \quad \text{by letting } \varepsilon\rightarrow 0.
\end{align*}

\textbf{Case 2.} If $1<p\leq q<2$, we get
\begin{equation*}
\begin{split}
&\quad \int_\Omega H(x,Du-Du_\varepsilon)\,dx\\
&=\int_\Omega\Big\{(|Du|+|Du_\varepsilon|)^\frac{p(2-p)}{2}(|Du|+|Du_\varepsilon|)^\frac{p(p-2)}{2}|Du-Du_\varepsilon|^p\\
&\quad\quad+a(x)(|Du|+|Du_\varepsilon|)^\frac{q(2-q)}{2}(|Du|+|Du_\varepsilon|)^\frac{q(q-2)}{2}|Du-Du_\varepsilon|^q\Big\}\,dx\\
&\leq \left(\int_\Omega(|Du|+|Du_\varepsilon|)^p\,dx\right)^\frac{2-p}{2}\left(\int_\Omega(|Du|+|Du_\varepsilon|)^{p-2}|Du-Du_\varepsilon|^2\,dx\right)^\frac{p}{2}\\
&\quad\quad+\left(\int_\Omega a(x)(|Du|+|Du_\varepsilon|)^q\,dx\right)^\frac{2-q}{2}\left(\int_\Omega a(x)(|Du|+|Du_\varepsilon|)^{q-2}|Du-Du_\varepsilon|^2\,dx\right)^\frac{q}{2}\\
&\leq \left[1+\left(\int_\Omega(|Du|+|Du_\varepsilon|)^p+a(x)(|Du|+|Du_\varepsilon|)^q\,dx\right)^\frac{2-p}{2}\right]\\
&\qquad\cdot\max_{t\in\{p,q\}}\left(\int_\Omega(|Du|+|Du_\varepsilon|)^{p-2}|Du-Du_\varepsilon|^2+a(x)(|Du|+|Du_\varepsilon|)^{q-2}|Du-Du_\varepsilon|^2\,dx\right)^\frac{t}{2}\\
&\leq C\left[1+\left(\int_\Omega H(x,Du)+H(x,Du_\varepsilon)\,dx\right)^\frac{2-p}{2}\right]\\
&\qquad\cdot\max_{t\in\{p,q\}}\left(\int_\Omega\langle A(x,Du_\varepsilon)-A(x,Du),Du_\varepsilon-Du\rangle\,dx\right)^\frac{t}{2}\\
&\leq C\left(1+\|Du\|^q_{L^{H(\cdot)}(\Omega)}+\|Du_\varepsilon\|^q_{L^{H(\cdot)}(\Omega)}\right)\\
&\qquad\cdot\max_{t\in\{p,q\}}\left(\int_\Omega\langle A(x,Du_\varepsilon)-A(x,Du),Du_\varepsilon-Du\rangle\,dx\right)^\frac{t}{2}\\
&\leq C\left(1+\|Du\|^{q^2}_{L^{H(\cdot)}(\Omega)}\right)\max_{t\in\{p,q\}}\left[\varepsilon C\left(1+\|Du\|^{q^2}_{L^{H(\cdot)}(\Omega)}\right)\right]^\frac{t}{2}\\
&\rightarrow 0 \quad \text{as } \varepsilon\rightarrow 0,
\end{split}
\end{equation*}
where in the penultimate inequality we employed \eqref{2-1} and the fact that $\frac{2-p}{2}<1$, and the last inequality follows from \eqref{4-1}.

\medskip

\textbf{Case 3.} $1<p<2\leq q<\infty$. Combining Case 1 and Case 2, we obtain
\begin{align*}
&\quad \int_\Omega H(x,Du-Du_\varepsilon)\,dx\\
&\leq \left(\int_\Omega(|Du|+|Du_\varepsilon|)^p\,dx\right)^\frac{2-p}{2}\left(\int_\Omega(|Du|+|Du_\varepsilon|)^{p-2}|Du-Du_\varepsilon|^2\,dx\right)^\frac{p}{2}\\
&\qquad+\int_\Omega a(x)\langle|Du|^{q-2}Du-|Du_\varepsilon|^{q-2}Du_\varepsilon,Du-Du_\varepsilon\rangle\,dx\\
&\leq \left[1+\left(\int_\Omega(|Du|+|Du_\varepsilon|)^p\,dx\right)^\frac{2-p}{2}\right]\\
&\qquad\cdot\max_{t\in\{p,2\}}\left(\int_\Omega\langle A(x,Du_\varepsilon)-A(x,Du),Du_\varepsilon-Du\rangle\,dx\right)^\frac{t}{2}\\
&\leq C\left(1+\int_\Omega (H(x,Du)+H(x,Du_\varepsilon))\,dx\right)\\
&\qquad\cdot\max_{t\in\{p,2\}}\left(\int_\Omega\langle A(x,Du_\varepsilon)-A(x,Du),Du_\varepsilon-Du\rangle\,dx\right)^\frac{t}{2}\\
&\leq C\left(1+\|Du\|^{q^2}_{L^{H(\cdot)}(\Omega)}\right)\max_{t\in\{p,2\}}\left[\varepsilon C\left(1+\|Du\|^{q^2}_{L^{H(\cdot)}(\Omega)}\right)\right]^\frac{t}{2}\\
& \rightarrow 0  \quad \text{when } \varepsilon\rightarrow 0.
\end{align*}

Consequently, we deduce
$$
\|Du-Du_\varepsilon\|_{L^{H(\cdot)}(\Omega)}\rightarrow0 \quad\text{as } \varepsilon\rightarrow0.
$$
By Lemma \ref{lem2-1}, when $\varepsilon\rightarrow0$,
$$
\|u-u_\varepsilon\|_{L^{H(\cdot)}(\Omega)}\rightarrow0.
$$
Therefore, we obtain
\begin{equation}
\label{4-3}
u_\varepsilon\rightarrow u \quad\text{in }  W^{1,H(\cdot)}(\Omega).
\end{equation}

Let $\varepsilon_1\leq\varepsilon_2$. Subtracting the corresponding equations gets
$$
\int_\Omega\langle A(x,Du_{\varepsilon_2})-A(x,Du_{\varepsilon_1}),D\eta\rangle\,dx=(\varepsilon_2-\varepsilon_1)\int_\Omega \eta\,dx\geq0,
$$
for any nonnegative $\eta\in C^\infty_0(\Omega)$. From Proposition \ref{pro2-4}, $u_{\varepsilon_2}\geq u_{\varepsilon_1}$ almost everywhere. This together with \eqref{4-3} indicates that
$$
u_\varepsilon\rightarrow u \quad\text{a.e. in } \Omega.
$$
It follows from the uniform $C^\alpha_{\rm loc}$-estimates for $u_\varepsilon$ in $\varepsilon$ that $u_\varepsilon\rightarrow u$ locally uniformly in $\Omega$.
\end{proof}

Before giving the main result of this section, we have to provide the following comparison principle for viscosity solutions which plays a crucial role in the equivalence of viscosity solutions and $\mathcal{A}_{H(\cdot)}$-harmonic functions.

\begin{lemma}
\label{lem4-3}
Assume that $u$ is a viscosity subsolution to Eq. \eqref{main}, and that $v$ is a distributional solution to Eq. \eqref{2-5} with local Lipschitz continuity. Under the assumption that $a(x)\equiv constant$, if $u\leq v$ on $\partial\Omega$, we then conclude
$$
u\leq v \quad\text{in }  \Omega.
$$
Analogously, the claim about viscosity supersolution $\widetilde{u}$ and locally Lipschitz continuous distributional solution $\widetilde{v}$ to
$$
-\dive A(x,D\widetilde{v})=-\varepsilon
$$
holds true as well. If $\widetilde{u}\geq\widetilde{v}$ on $\partial\Omega$, then
$$
\widetilde{u}\geq\widetilde{v} \quad\text{in }  \Omega.
$$
\end{lemma}
\begin{remark}
According to the proof of this lemma in Section \ref{sec-5}, we can derive the conditions imposed on $a(x)$ to ensure that  the comparison principle of viscosity solutions holds true. That is, for any $x\in \Omega$, it holds that $|a(x)-a(y)|=\mathcal{O}(|x-y|^{1+\sigma})$ as $y\rightarrow x$, where $\sigma>0$ is an arbitrary real number. In fact, the constraint on coefficient $a(x)$ implies that $a(x)\equiv constant$. Let us postpone the precise proof to next section.
\end{remark}

\begin{theorem}
\label{thm4-4}
When the coefficient $a(x)\equiv constant$, we can infer that the viscosity supersolutions to Eq. \eqref{main} are $\mathcal{A}_{H(\cdot)}$-superharmonic functions.
\end{theorem}


\begin{proof}
Let $u$ be a viscosity supersolution. To verify $u$ is a $\mathcal{A}_{H(\cdot)}$-superharmonic function, through Definition \ref{def2-2}, it is enough to prove that $u$ satisfies the comparison principle for distributional solutions to \eqref{main}. To this aim, we suppose $h\in C(\overline{D})$ is such a distributional solution of \eqref{main} that $u\geq h$ on $\partial D$, where $D\subset\subset\Omega$ is a subdomain. According to the lower semicontinuity of $u$, we can find that, for each $\delta>0$ and a smooth domain $D'\subset\subset D$ such that $h\leq u+\delta$ in $D\setminus D'$.

Now we are going to show $h\leq u+\delta$ in $D'$. If this is true, then we arrive at $h\leq u+\delta$ in $D$ and further by letting $\delta\rightarrow0$, we reach the conclusion that $h\leq u$ in $D$.

Consider the following Dirichlet problem
\begin{equation*}
\begin{cases}
-\dive A(x,Dh_\varepsilon)=-\varepsilon \quad\textmd{in } D', \\[2mm]
h_\varepsilon-h\in W^{1,H(\cdot)}_0(D').
\end{cases}
\end{equation*}
Denote the distributional solution by $h_\varepsilon$. Then $h_\varepsilon$ is locally Lipschitz continuous in $D'$ (see \cite{BCM18,CM15}). Moreover, owing to the smoothness of $D'$, $u+\delta\geq h_\varepsilon$ on $\partial D'$. Via Lemma \ref{lem4-3} and noting that $u+\delta$ is also a viscosity supersolution of \eqref{main}, we apply Lemma \ref{lem4-2} to infer
$$
u+\delta\geq h \quad \text{in }  D'.
$$
We now complete the proof.
\end{proof}

Finally, we end this section by the equivalence of the viscosity solutions and $\mathcal{A}_{H(\cdot)}$-harmonic functions. Combining Theorems \ref{thm4-1} and \ref{thm4-4}, we conclude the result below.

\begin{corollary}
\label{cor4-5}
When $a(x)\equiv constant$, viscosity solutions to Eq. \eqref{main} coincide with $\mathcal{A}_{H(\cdot)}$-harmonic functions.
\end{corollary}
\begin{remark}
In this section, we only prove the equivalence between the viscosity supersolutions and $\mathcal{A}_{H(\cdot)}$-superharmonic functions. Indeed, according to the Definitions \ref{def2-2} and \ref{def2-3}, the case of equivalence between viscosity subsolutions and $\mathcal{A}_{H(\cdot)}$-subharmonic functions is similar.
\end{remark}

\begin{remark}
By means of the main results coming from  Sections \ref{sec-3} and \ref{sec-4}, we can gain the equivalence of the (locally) bounded viscosity solutions and the (locally) bounded distributional solutions, under the assumption that $a(x)\equiv constant$.
\end{remark}

\section{The comparison principle}
\label{sec-5}
We in this section present the proof of Lemma \ref{lem4-3} that is the key ingredient to prove Theorem \ref{thm4-4}. We rewrite Lemma \ref{lem4-3}  as follows for the  readability and completeness.

\begin{proposition}
\label{pro5-1}
Suppose that for any $x\in \Omega$, it holds that $|a(x)-a(y)|=\mathcal{O}(|x-y|^{1+\sigma})$ as $y\rightarrow x$, where $\sigma>0$ is an arbitrary real number. Assume that $u$ is a viscosity subsolution to
$$
-\dive A(x,Du)=0 \quad\text{in } \Omega,
$$
and that $v$ is a distributional solution of
$$
-\dive A(x,Dv)=\varepsilon \quad (\varepsilon>0)\quad \text{in }  \Omega
$$
with local Lipschitz continuity. If $u\leq v$ on $\partial \Omega$, then we could infer
$$
u\leq v \quad\text{in } \Omega.
$$
\end{proposition}
\begin{remark}
Such restriction on the coefficient $a(x)$ in this proposition indicates that $a(x)$ is a constant function ($a(x)\equiv constant$) virtually.
But for the convenience of understanding why we impose the assumption that $a(x)\equiv constant$, we write it in the previous form.
\end{remark}

\begin{proof}
Let us prove it by contradiction. If the claim is not true, then we have
$$
0<\sup_\Omega\,(u-v)=:u(x_0)-v(x_0)
$$
for some $x_0\in \Omega$.

Consider the function
$$
\Psi_j(x,y):=u(x)-v(y)-\Phi_j(x,y),
$$
where $\Phi_j(x,y)=\frac{j}{s}|x-y|^s$ with $s>\max\{2,\frac{p}{p-1},\frac{q}{q-1}\}$. Let $(x_j,y_j)\in \overline{\Omega}\times\overline{\Omega}$ be the maximum point of $\Psi_j(x,y)$, that is
$$
\Psi_j(x_j,y_j)=\max_{\overline{\Omega}\times\overline{\Omega}}\Psi_j(x,y).
$$
We can verify that $(x_j,y_j)\in\Omega\times\Omega$ for $j$ large enough and
$$
(x_j,y_j)\rightarrow(x_0,x_0)
$$
by sending $j\rightarrow\infty$ (see \cite[Lemma 7.2]{CIL92}).

Now we first show $x_j\neq y_j$. By the definition of $(x_j,y_j)$, we get
$$
u(x_j)-v(y)-\Phi_j(x_j,y)\leq u(x_j)-v(y_j)-\Phi_j(x_j,y_j)
$$
for all $y\in \Omega$.
Then
$$
v(y)\geq -\Phi_j(x_j,y)+\Phi_j(x_j,y_j)+v(y_j)=:\phi_j(y).
$$
That is to say, $\phi_j(y)$ touches $v(y)$ at $y_j$ from below. Thus by Lemma \ref{lem2-6},
\begin{equation}
\label{5-1}
\limsup_{\stackrel{y\rightarrow y_j}{y\neq y_j}}\, (-\dive A(y,D\phi_j(y)))\geq \varepsilon.
\end{equation}
By expanding $\dive A(y,D\phi_j(y))$, we obtain that
\begin{align*}
&\quad\dive A(y,D\phi_j(y))\\
&=|D\phi_j(y)|^{p-2}\left(\mathrm{tr}D^2\phi_j(y)+(p-2)\left\langle D^2\phi_j(y)\frac{D\phi_j(y)}{|D\phi_j(y)|},\frac{D\phi_j(y)}{|D\phi_j(y)|}\right\rangle\right)\\
&\quad+a(y)|D\phi_j(y)|^{q-2}\left(\mathrm{tr}D^2\phi_j(y)+(q-2)\left\langle D^2\phi_j(y)\frac{D\phi_j(y)}{|D\phi_j(y)|},\frac{D\phi_j(y)}{|D\phi_j(y)|}\right\rangle\right)\\
&\quad+|D\phi_j(y)|^{q-2}D\phi_j(y)\cdot Da(y).
\end{align*}
Here we denote the trace of matrix $M$ by $\mathrm{tr}M$. Furthermore,
$$
D\phi_j(y)=-D\Phi_j(x_j,y)=j|x_j-y|^{s-2}(x_j-y),
$$
$$
D^2\phi_j(y)=-D^2\Phi_j(x_j,y)=-j|x_j-y|^{s-2}I-j(s-2)|x_j-y|^{s-4}(x_j-y)\otimes(x_j-y),
$$
where $\xi\otimes\xi$ denotes the matrix with entries $\xi_i\xi_j$ for $\xi\in\mathbb{R}^n$.
Therefore, we have
\begin{align*}
-\dive A(y,D\phi_j(y))&=j^{p-1}[n+s-2+(p-2)(s-1)]|x_j-y|^{(p-2)(s-1)+(s-2)}\\
&\quad+a(y)j^{q-1}[n+s-2+(q-2)(s-1)]|x_j-y|^{(q-2)(s-1)+(s-2)}\\
&\quad-j^{q-1}|x_j-y|^{(q-2)(s-1)+(s-2)}(x_j-y)\cdot Da(y).
\end{align*}
Note that due to $s>\max\{2,\frac{p}{p-1},\frac{q}{q-1}\}$,
\begin{align*}
&(p-2)(s-1)+(s-2)=s(p-1)-p>0,\\
&(q-2)(s-1)+(s-2)=s(q-1)-q>0.
\end{align*}
Thus if $x_j=y_j$, then
$$
\lim_{y\rightarrow y_j}-\dive A(y,D\phi_j(y))=0,
$$
which contradicts \eqref{5-1}.

Via Theorem of Sums in \cite{CIL92}, for each $\mu>0$, there are symmetric $n\times n$ matrices $X:=X(\mu)$ and $Y:=Y(\mu)$ such that
\begin{align*}
&\left(D_x\Phi_j(x_j,y_j),X\right)\in \overline{J}^{2,+}u(x_j),\\
&\left(-D_y\Phi_j(x_j,y_j),Y\right)\in \overline{J}^{2,-}v(y_j)
\end{align*}
and
\begin{equation*}
\begin{split}
-(\mu+\|D^2\Phi_j(x_j,y_j)\|)\left(\begin{array}{cc}
I& \\[2mm]
 &I
\end{array}
\right)&\leq
\left(\begin{array}{cc}
X & \\[2mm]
 &-Y
\end{array}
\right)\\
&\leq D^2\Phi_j(x_j,y_j)+\frac{1}{\mu}(D^2\Phi_j(x_j,y_j))^2,
\end{split}
\end{equation*}
where
\begin{equation*}
D^2\Phi_j(x_j,y_j)=\left(\begin{array}{cc}
D_{xx}\Phi_j(x_j,y_j) &D_{xy}\Phi_j(x_j,y_j) \\[2mm]
D_{yx}\Phi_j(x_j,y_j) &D_{yy}\Phi_j(x_j,y_j)
\end{array}\right).
\end{equation*}
By direct calculation,
\begin{equation*}
D^2\Phi_j(x_j,y_j)=:\left(\begin{array}{cc}
B &-B\\[2mm]
-B &B
\end{array}\right),
\end{equation*}
where $B=j|z_j|^{s-2}I+(s-2)j|z_j|^{s-4}z_j\otimes z_j$ with $z_j=x_j-y_j$.
After manipulation,
\begin{equation*}
\begin{split}
-(\mu+2\|B\|)\left(\begin{array}{cc}
I& \\[2mm]
 &I
\end{array}
\right)&\leq
\left(\begin{array}{cc}
X & \\[2mm]
 &-Y
\end{array}
\right)\\[2mm]
&\leq \left(\begin{array}{cc}
B &-B\\[2mm]
-B &B
\end{array}\right)+\frac{2}{\mu}\left(\begin{array}{cc}
B^2 &-B^2\\[2mm]
-B^2 &B^2
\end{array}\right).
\end{split}
\end{equation*}
It is easy to derive that $X\leq Y$, that is $\langle(X-Y)\xi,\xi\rangle\leq0$ for any $\xi\in\mathbb{R}^n$. Moreover, we have
$$
\|X\|,\|Y\|\leq 2\|B\|+\mu.
$$
For more details, we refer the readers to \cite{IL90}. In the rest of the proof, we take a special value of $\mu$, i.e., $\mu=j|x_j-y_j|^{s-2}$.

Let
$$
\eta_j:=D_x\Phi_j(x_j,y_j)=-D_y\Phi_j(x_j,y_j)=j|x_j-y_j|^{s-2}(x_j-y_j).
$$
Notice that $\eta_j\neq0$, which is of great importance. We now define
\begin{align*}
F(x,\eta,X)&=F_1(x,\eta,X)+F_2(x,\eta,X)+F_3(x,\eta,X)\\
&:=-|\eta|^{p-2}\left(\mathrm{tr}X+(p-2)\left\langle X\frac{\eta}{|\eta|},\frac{\eta}{|\eta|}\right\rangle\right)\\
&\quad+\left(-a(x)|\eta|^{q-2}\left(\mathrm{tr}X+(q-2)\left\langle X\frac{\eta}{|\eta|},\frac{\eta}{|\eta|}\right\rangle\right)\right)\\
&\quad+(-|\eta|^{q-2}\eta\cdot Da(x)),
\end{align*}
for $x\in \Omega, \eta\in\mathbb{R}^n$ and $X$ being $n\times n$ symmetric matrix.
It is well known that $u$ is a viscosity subsolution of \eqref{main}, and $v$ is a viscosity supersolution of \eqref{2-5}, so we can arrive at
\begin{align*}
&F(x_j,\eta_j,X_j)\leq0, \\
&F(y_j,\eta_j,Y_j)\geq\varepsilon,
\end{align*}
where $X_j:=X(j|x_j-y_j|^{s-2})$ and $Y_j:=Y(j|x_j-y_j|^{s-2})$. Subtracting these two inequalities yields that
\begin{equation}
\label{5-2}
\begin{split}
\varepsilon &\leq F(y_j,\eta_j,Y_j)-F(x_j,\eta_j,X_j)\\
&=F_1(y_j,\eta_j,Y_j)-F_1(x_j,\eta_j,X_j)\\
&\quad+F_2(y_j,\eta_j,Y_j)-F_2(x_j,\eta_j,X_j)\\
&\quad+F_3(y_j,\eta_j,Y_j)-F_3(x_j,\eta_j,X_j).
\end{split}
\end{equation}
Since $\Psi_j(x,y)$ reaches the maximum at $(x_j,y_j)\in \Omega\times\Omega$, we can deduce by the local Lipschitz continuity of $v$ that
\begin{align*}
&u(x_j)-v(x_j)\leq u(x_j)-v(y_j)-\Phi_j(x_j,y_j)\\
\Rightarrow & \ \frac{j}{s}|x_j-y_j|^s\leq v(x_j)-v(y_j)\leq C|x_j-y_j|.
\end{align*}
Hence for any $\sigma>0$,
$$
j|x_j-y_j|^{s-1+\sigma}\leq C|x_j-y_j|^\sigma.
$$
We further know
\begin{equation}
\label{5-3}
j|x_j-y_j|^{s-1+\sigma}\rightarrow0 \quad (j\rightarrow\infty)
\end{equation}
for all $\sigma>0$. Furthermore,
\begin{equation}
\label{5-4}
j|x_j-y_j|^{s-1}\leq C,
\end{equation}
where $C$ is independent of $j$.

We proceed by estimating the three terms on the right-hand side in \eqref{5-2}, respectively. First, we observe $F_1(x,\eta,X)$ does not rely on $x$-variable indeed and is monotone decreasing with respect to $X$-variable, so by $X_j\leq Y_j$ we get
\begin{equation}
\label{5-5}
F_1(y_j,\eta_j,Y_j)-F_1(x_j,\eta_j,X_j)=F_1(\eta_j,Y_j)-F_1(\eta_j,X_j)\leq0.
\end{equation}
Second, we evaluate
\begin{equation}
\label{5-6}
\begin{split}
F_3(y_j,\eta_j,Y_j)-F_3(x_j,\eta_j,X_j)&=|\eta_j|^{q-2}\eta_j\cdot(Da(x_j)-Da(y_j))\\
&\leq |\eta_j|^{q-1}|Da(x_j)-Da(y_j)|\rightarrow0,
\end{split}
\end{equation}
as $j\rightarrow\infty$. In fact, $|\eta_j|=j|x_j-y_j|^{s-1}\leq C$ and $a(x)\in C^1(\Omega)$, $x_j,y_j\rightarrow x_0$.

Finally, we deal with the second term $F_2(y_j,\eta_j,Y_j)-F_2(x_j,\eta_j,X_j)$. According to the value of $\mu$ and $\|B\|\leq (s-1)j|x_j-y_j|^{s-2}$, we derive
\begin{equation}
\label{5-7}
\|X_j\|\leq (2s-1)j|x_j-y_j|^{s-2}.
\end{equation}
Now by the restriction on $a(x)$ in Proposition \ref{pro5-1} and inequalities \eqref{5-5} and \eqref{5-7}, we estimate
\begin{equation}
\label{5-8}
\begin{split}
& \quad F_2(y_j,\eta_j,Y_j)-F_2(x_j,\eta_j,X_j)\\
&=a(x_j)|\eta_j|^{q-2}\left(\mathrm{tr}X_j+(q-2)\left\langle X_j\frac{\eta_j}{|\eta_j|},\frac{\eta_j}{|\eta_j|}\right\rangle\right)\\
&\quad -a(y_j)|\eta_j|^{q-2}\left(\mathrm{tr}Y_j+(q-2)\left\langle Y_j\frac{\eta_j}{|\eta_j|},\frac{\eta_j}{|\eta_j|}\right\rangle\right)\\
&=(a(x_j)-a(y_j))|\eta_j|^{q-2}\left(\mathrm{tr}X_j+(q-2)\left\langle X_j\frac{\eta_j}{|\eta_j|},\frac{\eta_j}{|\eta_j|}\right\rangle\right)\\
&\quad+a(y_j)(F_1(y_j,\eta_j,Y_j)-F_1(x_j,\eta_j,X_j))\\
&\leq (a(x_j)-a(y_j))|\eta_j|^{q-2}\left(\mathrm{tr}X_j+(q-2)\left\langle X_j\frac{\eta_j}{|\eta_j|},\frac{\eta_j}{|\eta_j|}\right\rangle\right)\\
&\leq C|a(x_j)-a(y_j)||\eta_j|^{q-2}\|X_j\|\\
&\leq C|\eta_j|^{q-2}|a(x_j)-a(y_j)|j|x_j-y_j|^{s-2}\\
&\leq Cj|x_j-y_j|^{s-1+\sigma}\rightarrow0 \quad \text{as } j\rightarrow \infty.
\end{split}
\end{equation}
Finally, Merging \eqref{5-2}, \eqref{5-5}, \eqref{5-6} and \eqref{5-8} deduces
$$
\varepsilon\leq F(y_j,\eta_j,Y_j)-F(x_j,\eta_j,X_j)\leq 0.
$$
It is a contradiction. We now complete the proof.
\end{proof}

\section*{Acknowledgments}

This work was supported by the NSFC  (No. 11671111).

\end{document}